\documentclass{amsart}
\usepackage{amsfonts}
\usepackage{amsmath}
\usepackage{amssymb}
\usepackage{graphicx}

\newtheorem {theorem}{Theorem}
\newtheorem {acknowledgement}{Acknowledgement}

\newtheorem {corollary}{Corollary}

\newtheorem {lemma}{Lemma}

\newtheorem {proposition}{Proposition}

\numberwithin {equation}{section}   
\input{tcilatex}

\begin{document}
\title[Positive exponential sums and odd polynomials]{Positive exponential
sums and odd polynomials}
\author{M. Nin\v{c}evi\'{c} and S. Slijep\v{c}evi\'{c}}
\address{Department of Mathematics, Bijeni\v{c}ka 30, Zagreb, Croatia}
\email{nincevic@math.hr, slijepce@math.hr}
\urladdr{}
\date{5 November 2012}

\begin{abstract}
Given an odd integer polynomial $f(x)$ of a degree $k\geq 3$, we construct a
non-negative valued, normed trigonometric polynomial with the spectrum in
the set of integer values of $f(x)$ not greater than $n$, and a small free
coefficient $a_{0}=O((\log n)^{-1/k})$. This gives an alternative proof for
the maximal possible cardinality of a set $A$, so that $A-A$ does not
contain an element of $f(x)$. We also discuss other interpretations and an
ergodic characterization of that bound.
\end{abstract}

\maketitle

\section{Introduction}

We consider polynomials $f(x)=\alpha _{k}x^{k}+...+\alpha _{1}x$ with
integer coefficients, satisfying:%
\begin{equation}
\text{For all }j\text{ even, }\alpha _{j}=0\text{;\ and the leading
coefficient is }\alpha _{k}>0\text{.}  \label{a:polynomial}
\end{equation}%
The main result of the paper is the following:

\begin{theorem}
\label{t:main1}Given an integer polynomial $f$ of a degree $k\geq 3$
satisfying (\ref{a:polynomial}), there exist cosine polynomials 
\begin{equation}
T(x)=b_{0}+\sum_{0<f(j)\leq N}b_{f(j)}\cos (2\pi f(j)x),  \label{r:type}
\end{equation}%
such that for all $x$, $T(x)\geq 0$, and such that all coefficients $b_{j}$
are non-negative, normed (i.e. $\sum b_{j}=1$), and that $b_{0}=O((\log
N)^{-1/k})$.
\end{theorem}

We now discuss the background and implications of that result. Denote by ${%
\mathcal{T}}(D)$ the set of all cosine polynomials with the spectrum in a
set of positive integers $D\cup \{0\}$, such that $T(x)\geq 0$, $T(0)=1$;\
and let ${\mathcal{T}}^{+}(D)$ be the subset of ${\mathcal{T}}(D)$ with
non-negative coefficients. Kamae and Mend\`{e}s France in \cite{Kamae:77}
introduced the notion of van der Corput sets (VdC sets;\ or correlative
sets), if $\func{inf}_{T\in {\mathcal{T}}(D)}b_{0}=0\,\ $($b_{0}$ is the
free coefficient). One can also define VdC$^{+}$ sets as those for which $%
\func{inf}_{T\in {\mathcal{T}}(D^{+})}b_{0}=0$. Let $\gamma (n)$, $\gamma
^{+}(n)$ be the arithmetic functions which measure how quickly a set is
becoming a van der Corput set: 
\begin{equation*}
\gamma (n)= \func{inf}_{T\in {\mathcal{T}}(D\cap \lbrack 1,n])}b_{0}
\end{equation*}
and $\gamma ^{+}(n)$ analogously for ${\mathcal{T}}^{+}(D)$ (and then $%
\gamma (n)\leq \gamma ^{+}(n)$). Theorem \ref{t:main1} can now be rephrased
as follows:\ for the sets of values of an odd integer polynomial $f$, 
\begin{equation}
\gamma ^{+}(n)=O((\log n)^{-1/k}).  \label{r:bound1}
\end{equation}

Kamae and Mend\`{e}s France, Ruzsa and Montgomery (\cite{Montgomery:94}, 
\cite{Ruzsa:81}, \cite{Ruzsa:84a}) described various characterizations of
van der Corput sets and the function $\gamma (n)$, mainly related to uniform
distribution properties of the set $D$. In particular, in \cite{Kamae:77} it
was shown that van der Corput sets are intersective sets, and Ruzsa showed
in \cite{Ruzsa:84a} that an upper bound for the function $\gamma $ is also
an upper bound for the intersective property. Following that, a Corollary of
Theorem \ref{t:main1} is the following:

\begin{corollary}
Let $f$ be an integer polynomial satisfying (\ref{a:polynomial}). Suppose
that N is an integer and that $A\subset \{1,...,N\}$ is such that the
difference between any two elements of A is never an integer value of $f$.
Then $|A|=O(N(\log N)^{-1/k})$.
\end{corollary}

This gives another proof of the upper bound for the difference property of
odd polynomials, where the best current result (by Lucier \cite{Lucier:07},
valid for all polynomials) is $N(\log N)^{-1/(k-1+o(1))}$.

Montgomery set a problem in \cite{Montgomery:94} for finding any upper
bounds for the van der Corput property for any "interesting" sets, such as
the set of squares and more generally the set of values of an integer
polynomial. Ruzsa in \cite{Ruzsa:81} announced the $\gamma (n)=O((\log
n)^{-1/2})$ bound for the set of squares, but the proof was never published.
One of the authors in \cite{Slijepcevic:10}, \cite{Slijepcevic:12} proved
bounds $\gamma (n)\leq \gamma ^{+}(n)=O((\log n)^{-1/3})$ for the set of
perfect squares, and $\gamma (n)\leq \gamma ^{+}(n)=O((\log n)^{-1+o(1)})$
for the set of shifted primes. We also note that Theorem \ref{t:main1} can
be extended to all integer polynomials of degree $k\geq 3$, but for now with
the bound only $\gamma ^{+}(n)=O((\log \log n)^{-1/k^{2}})$ (\cite%
{Nincevic:12}).

Ruzsa showed that, by using only non-negative coefficients in the case of
squares, one can not do better than $O((\log n)^{-1})$. We extend the same
argument to show that the Theorem \ref{t:main1} is close to optimal if only
non-negative coefficients are used:

\begin{theorem}
\label{t:lower}Let $f(x)=\beta x^{k}$, $\beta >0$ an integer, and $k\geq 3$
an odd integer. Then 
\begin{equation}
\gamma ^{+}(n)\geq (1/\varphi (k)+o(1))(\log n)^{-1}.  \label{r:bound2}
\end{equation}
\end{theorem}

It is hoped that one can improve the van der Corput and intersective sets
bounds by constructing cosine polynomials also with negative coefficients.
Ruzsa and Matolcsi have recently announced progress in this direction in the
case of perfect squares;\ and also discussed this in a more general setting
of commutative groups (\cite{Matolcsi:12a}).

The Theorems \ref{t:main1}, \ref{t:lower} have an ergodic-theoretical
interpretation, as was noted in \cite{Rabar:12}:

\begin{corollary}
Let $f$ be an integer polynomial of a degree $k\geq 3$ satisfying (\ref%
{a:polynomial}), $H$ an arbitrary real Hilbert space, $U$ an unitary
operator on $H$, and $P$ the projection to the kernel of $U-I$. If $x\in H$
is such that $Px\not=0$, then there exists a positive integer $f(j)$ such
that $(U^{f(j)}x,x)>0$.

Furthermore, if $(Px,x)>\gamma ^{+}(n)(x,x)$, then there exists such $%
f(j)\leq n$, where $\gamma ^{+}(n)$ is the best such bound valid universally
for all $H,U$, with bounds (\ref{r:bound1}), and (\ref{r:bound2}) in the
case $f(x)=\alpha _kx^{k}$.
\end{corollary}

\textbf{The structure of the paper. }We first introduce some notation
related to the polynomial $f$. The degree of $f$ will be always denoted by $%
k $, and let $l\geq 1$ be the smallest index so that $\alpha _{l}>0$. Let $%
c(f)=(\alpha _{k},...,\alpha _{l})$ be the content of the polynomial.
Without loss of generality we always assume that for $x\geq 1$, $f(x)\geq 1$%
, and that $f(j)$, $j\geq 1$ is an increasing sequence (if not, we find the
smallest $j_{0}$ such that it holds for $j\geq j_{0}$, and modify all the
estimates by skipping the first $j_{0}$ values of $f$, this impacting only
the implicit constants in the estimates).

Let $F_{n}(x)$ be the normed Fej\'{e}r's kernel 
\begin{equation*}
F_{n}(x)=\frac{1}{n}+2\sum_{j=1}^{n}\left( \frac{1}{n}-\frac{j}{n^{2}}%
\right) \cos (2\pi jx){\text{,}}
\end{equation*}%
and then $F_{n}(x)\geq 0$, and $F_{n}(0)=1$. The key tool in our
construction will be, following the idea of I. Ruzsa, construction of a
polynomial of the type (\ref{r:type}) which approximates $F_{n}(x)$. We may
further restrict "allowable"\ indices $j$ to those with an integer $d$ as a
factor, and define 
\begin{equation}
G_{n,d}(x)=\frac{2}{K}\sum_{\alpha _{k}(dj)^{k}\leq n}\alpha
_{k}kd^{k}j^{k-1}\left( \frac{1}{n}-\frac{\alpha _{k}(dj)^{k}}{n^{2}}\right)
\cos (2\pi f(dj)x),  \label{r:defg}
\end{equation}%
where $K$ is chosen so that $G_{n,d}(0)=1$ ($K$ will be close to $1$ for $n$
large enough;\ and will be estimated in Section 2).

The structure of the proof is as follows:\ let $S(f,q)$ be the complete
trigonometric sum 
\begin{equation}
S(f,q)=\sum_{s=0}^{q-1}e(f(s)/q),  \label{r:complete}
\end{equation}%
where $e(x)=\func{exp}(2\pi ix)$. We will need the reduced complete
trigonometric sum over the multipliers of an integer $d$: 
\begin{equation*}
S_{d}(f,q)=\sum_{s=0}^{q-1}e(f(ds)/q).
\end{equation*}%
For $q$ small (the major arc estimates), we show that 
\begin{equation}
G_{n,d}(a/q+\kappa )=\frac{1}{q}S_{d}(af,q)F_{n}(\kappa )+{\text{error term,}%
}  \label{r:key}
\end{equation}%
where the error term is small for small\ $\kappa $ and large $n$ as compared
to $q,d$. For large $q$, we show by partial summation and by using the
well-known Vinogradov's trigonometric sum estimates that $G_{n,d}(x)$ is
small. The key step is the averaging step:\ we choose the constants $%
d_{0},...,d_{s}$ and normalized weights $w_{0},\ldots ,w_{s}$ such that for
any $q$, $\sum_{j}w_{j}S_{d_{j}}(af,q)\geq -\delta $, $\delta =O((\log
N)^{-1/k})$. Here $N$ is the size of the largest non-zero coefficient in the
family of polynomials $G_{n,d_{j}}$. The estimate $b_{0}=O((\log N)^{-1/k})$
follows from this and the size of the error term.

Unfortunately, for polynomials which are not odd, this approach seems to
fail as (\ref{r:key}) does not hold. Namely, there is another factor
difficult to control if one can not a-priori claim that the imaginary part
of $S(f,q)$ is $0$, as is the case for odd polynomials.

We prove Theorem \ref{t:main1} in Sections 2-5, and Theorem \ref{t:lower} in
Section 6.

\section{The major arcs}

We will use the notation $O_{k}(.)$, $\ll _{k}$, $O_{f}(.)$, $\ll _{f}$when
the implicit constant depends implicitly on the degree $k$ or the
coefficients of the polynomial $f$ (including the degree) respectively. We
will often use the following relations:\ If $x,y\geq 1$ are integers such
that $f(x),f(y)\ll _fn$, then 
\begin{eqnarray}
f(x)&=&O_{f}(x^{k})=  \label{r:pol1} \\
&=&\alpha _{k}x^{k}+O_{f}(n^{1-1/k}),  \label{r:pol1A} \\
|f(x)-f(y)|&=&|x-y|O_{f}(n^{1-1/k})=  \label{r:pol2} \\
&=&\alpha _{k}k|x-y|x^{k-1}+|x-y|^{2}O_{f}(n^{1-2/k}).  \label{r:pol2A}
\end{eqnarray}
The relations above can be computed easily by using $\beta =2(|\alpha
_{k-1}|+\cdots +|\alpha _1|)/\alpha _k$ and the relation $\alpha _kx^k\leq
2f(x)$ if $x\geq 1$ and $x\geq \beta $.

The following result by Chen \cite{Chen:77} and Nechaev \cite{Nechaev:75}
gives a bound for the complete trigonometric sums.

\begin{lemma}
\label{l:cn}Let $f$ be an integer polynomial of a degree $k\geq 2$. Then 
\begin{equation*}
S(f,q)=O_{k}((c(f),q)^{1/k}q^{1-1/k}).
\end{equation*}
\end{lemma}

We will need the next bound for the content of the polynomial $f$ when it
goes over the multipliers of an integer $d$:

\begin{lemma}
\label{l:content}Let $f$ be an integer polynomial of a degree $k\geq 1$, $d$
an integer and $g(x)=f(dx)$, for all $x$. Then 
\begin{equation*}
c(g)\leq d^{l}|\alpha _{l}|^{k}.
\end{equation*}
\end{lemma}

\begin{proof}
We first assume that $f$ is a primitive polynomial (i.e. $c(f)=1$). One can
than easily show that $c(g)\leq d^{l}(d,\alpha _{l})^{k-l}.$ If $f$ is not
primitive, we apply the previous result to the polynomial $f/c(f)$ and get 
\begin{equation*}
c(f/c(f))\leq d^{l}(d,\alpha _{l}/c(f))^{k-l}\leq d^{l}|\alpha _{l}|^{k-l}.
\end{equation*}%
The claim now follows from $c(g)=c(f/c(f))c(f)$ and $c(f)\leq |\alpha _{l}|$.
\end{proof}

We now state the major arcs estimate.

\begin{proposition}
Let $G_{n,d}(x)$ be a trigonometric polynomial as in (\ref{r:defg}) for some
integer polynomial $f$ of a degree $k\geq 3$ satisfying (\ref{a:polynomial}%
), and $n,d$ positive integers. Let $x=a/q+\kappa $. Then 
\begin{equation*}
G_{n,d}(x)=\frac{1}{q}S_{d}(af,q)F_{n}(\kappa )+O_{f}(dqn^{-1/k}(1+|\kappa
|n)){\text{.}}
\end{equation*}
\end{proposition}

\begin{proof}
Without loss of generality we assume that $dqn^{-1/k}\leq \alpha _{k}^{-1/k}$
(otherwise the error term is of the order $1$ and the claim is trivial). By
writing ${\mathrm{Re}}\,e(y)$ instead of $\cos (2\pi y)$ and appropriate
grouping, we get 
\begin{equation*}
G_{n,d}(x)={\mathrm{Re}}\,\sum_{s=0}^{q-1}\frac{1}{q}e(f(ds)a/q)\frac{2}{K}%
\sum_{\substack{ \alpha _{k}(dj)^{k}\leq n  \\ j\equiv s({\mathrm{mod}}\,q)}}%
\alpha _{k}qkd^{k}j^{k-1}\left( \frac{1}{n}-\frac{\alpha _{k}(dj)^{k}}{n^{2}}%
\right) e(f(dj)\kappa ).
\end{equation*}%
We fix all the parameters and constants. Let $B/2$ be the inner sum in the
expression above, $A=B/K,$ and let 
\begin{eqnarray*}
C &=&2\sum_{\substack{ \alpha _{k}(dj)^{k}\leq n  \\ j\equiv s({\mathrm{mod}}%
\,q)}}e(f(dj)\kappa )\sum_{t=f(dj)}^{f(d(j+q))-1}\left( \frac{1}{n}-\frac{t}{%
n^{2}}\right) , \\
D &=&2\sum_{\substack{ f(dj)\leq n  \\ j\equiv s({\mathrm{mod}}\,q)}}%
e(f(dj)\kappa )\sum_{t=f(dj)}^{f(d(j+q))-1}\left( \frac{1}{n}-\frac{t}{n^{2}}%
\right) , \\
F_{n}^{\ast }(\kappa ) &=&\frac{1}{n}+2\sum_{t=1}^{n}\left( \frac{1}{n}-%
\frac{t}{n^{2}}\right) e(t\kappa ).
\end{eqnarray*}%
Note that $F_{n}(x)={\mathrm{Re}}\,F_{n}^{\ast }(x)$, but $F_{n}^{\ast }$
has in general a non-zero imaginary part.

If $m$ is chosen so that $\alpha _{k}(dm)^{k}\leq n<\alpha _{k}(d(m+1))^{k}$%
, then one easily gets 
\begin{equation*}
|\alpha _{k}d^{k}m^{k}-n|\leq \alpha _{k}d^{k}((m+1)^{k}-m^{k})\leq \alpha
_{k}k2^{k-1}d(dm)^{k-1}.
\end{equation*}%
We now have%
\begin{equation}
\alpha _{k}d^{k}m^{k}=n+O_{f}(dn^{1-1/k}).  \label{r:estimatemm}
\end{equation}%
Using that and $\sum_{1\leq j\leq m}j^{k-1}=(1/k)m^{k}+O_{k}(m^{k-1})$, we
estimate $K$: 
\begin{eqnarray}
K &=&2\sum_{\alpha _{k}(dj)^{k}\leq n}\alpha _{k}kd^{k}j^{k-1}\left( \frac{1%
}{n}-\frac{\alpha _{k}(dj)^{k}}{n^{2}}\right) =  \notag \\
&=&\frac{2}{n}\alpha _{k}d^{k}(m^{k}+O_{k}(m^{k-1}))-\frac{1}{n^{2}}\alpha
_{k}^{2}d^{2k}(m^{2k}+O_{k}(m^{2k-1}))=  \notag \\
&=&1+O_{f}(dn^{-1/k}){\text{.}}  \label{r:DD}
\end{eqnarray}%
Similarly, by using the elementary fact that 
\begin{equation}
\sum_{\substack{ 1\leq j\leq m  \\ j\equiv s({\mathrm{mod}}\,q)}}%
j^{k-1}=O_{k}\left( \frac{1}{q}m^{k}+m^{k-1}\right) {\text{,}}
\label{r:sumpol}
\end{equation}%
one gets that $|B|\leq 1+O_{f}(dqn^{-1/k})$. The assumption $dqn^{-1/k}\ll
_{f}1$ implies $B=O_{f}(1)$, thus 
\begin{equation}
A-B=B(1/K-1)=O_{f}(dn^{-1/k}){\text{.}}  \label{r:AB}
\end{equation}

If $1\leq j\leq m$ and $f(dj)\leq t\leq f(d(j+q))-1$, then (\ref{r:pol1A})\
and (\ref{r:pol2}) imply that 
\begin{equation*}
t=\alpha _{k}(dj)^{k}+O_{f}(dqn^{1-1/k}){\text {.}}
\end{equation*}
Using that and (\ref{r:pol2A}), we get 
\begin{eqnarray}
&&\sum _{t=f(dj)}^{f(d(j+q))-1}\left (\frac{1}{n}-\frac{t}{n^{2}}\right )= 
\notag \\
&=&(f(d(j+q))-f(dj))\left (\frac{1}{n}-\frac{\alpha _{k}(dj)^{k}}{n^{2}}+%
\frac{O_{f}(dqn^{1-1/k})}{n^{2}}\right )=  \notag \\
&=&\left (\alpha _{k}qkd^{k}j^{k-1}+O_{f}(d^{2}q^{2}n^{1-2/k})\right )\left (%
\frac{1}{n}-\frac{\alpha _{k}(dj)^{k}}{n^{2}}+O_{f}(dqn^{-1-1/k})\right )= 
\notag \\
&=&\alpha _{k}qkd^{k}j^{k-1}\left (\frac{1}{n}-\frac{\alpha _{k}(dj)^{k}}{%
n^{2}}\right )+\alpha
_{k}qkd^{k}j^{k-1}O_{f}(dqn^{-1-1/k})+O_{f}(d^{2}q^{2}n^{-2/k}){\text {.}} 
\notag
\end{eqnarray}
By summing the previous over all the summands "$j$" in the definition of $C$%
, we get 
\begin{eqnarray}
|B-C|&=&O_{f}(dqn^{-1-1/k})\sum _{\substack{ \alpha _{k}(dj)^{k}\leq n  \\ %
j\equiv s({\mathrm{mod}}\,q)}}\alpha
_{k}qkd^{k}j^{k-1}+O_{f}(d^{2}q^{2}n^{-2/k})\sum _{\substack{ \alpha
_{k}(dj)^{k}\leq n  \\ j\equiv s({\mathrm{mod}}\,q)}}1,  \notag
\end{eqnarray}
Now (\ref{r:sumpol}) implies 
\begin{eqnarray}
B-C=O_{f}(dqn^{-1/k}){\text {.}}  \label{r:BBCC}
\end{eqnarray}

From (\ref{r:estimatemm}), as $\alpha _{k}d^{k}(m+q)^{k}-\alpha
_{k}d^{k}m^{k}\leq \alpha _{k}d^{k}qk(m+q)^{k-1}$ and $d(m+q)\ll _{f}n^{1/k}$%
, we get $\alpha _{k}d^{k}(m+q)^{k}=n+O_{f}(dqn^{1-1/k})$. Therefore (\ref%
{r:pol1A}) implies 
\begin{equation}
f(d(m+q))=n+O_{f}(dqn^{1-1/k}).  \label{r:CD1}
\end{equation}%
Choose $m_{\ast }$ so that $f(dm_{\ast })\leq n<f(d(m_{\ast }+1))$. Assume
that $m_{\ast }\leq m$ (the second case is proved analogously). If $%
f(d(m_{\ast }+1))\leq t\leq f(d((m+q))-1$, then (\ref{r:CD1}) implies 
\begin{equation}
t=n+O_{f}(dqn^{1-1/k}).  \label{r:CD2}
\end{equation}%
Similarly as before, one shows that 
\begin{equation}
f(d(m_{\ast }+1))=n+O_{f}(dn^{1-1/k})  \label{r:CD3}
\end{equation}%
and 
\begin{equation}
f(d(m_{\ast }+q))=n+O_{f}(dqn^{1-1/k}).  \label{r:CD4}
\end{equation}%
Now, $C$ and $D$ only differ in the number of summands $(1/n-t/n^{2})$, thus
by (\ref{r:CD2}), one gets 
\begin{equation*}
|C-D|\leq 2\sum_{t=f(d(m_{\ast }+1))}^{f(d((m+q))-1}\left\vert \frac{1}{n}-%
\frac{t}{n^{2}}\right\vert =O_{f}(dqn^{-1-1/k})(f(d(m+q))-f(d(m_{\ast }+1))){%
\text{.}}
\end{equation*}%
Using (\ref{r:CD1}) and (\ref{r:CD3}), we deduce that 
\begin{equation}
C-D=O_{f}(dqn^{-1/k}).  \label{r:CCDD}
\end{equation}

We now estimate $D-F_{n}^{\ast }(\kappa )$. If $1\leq j\leq m_*$ and $%
f(dj)\leq t\leq f(d(j+q))-1$, the relations $|e(x)-e(y)|\leq 2\pi |x-y|$ and
(\ref{r:pol2}) imply 
\begin{equation*}
|e(f(dj)\kappa )-e(t\kappa )|\leq 2\pi |\kappa
|(f(d(j+q))-f(dj))=O_{f}(|\kappa |dqn^{1-1/k}){\text {.}}
\end{equation*}
Comparing $D$ and $F_{n}^{\ast }(\kappa )$, we see that 
\begin{eqnarray}
|D-F_{n}^{\ast }(\kappa )|&\leq &\frac{1}{n}+2\sum _{t=1}^{f(ds)-1}\left |%
\frac{1}{n}-\frac{t}{n^{2}}\right ||e(t\kappa )|  \notag \\
&&+2\sup _{\substack{ 1\leq j\leq m_*,\,j\equiv s\,({\mathrm{mod}}\,q)  \\ %
f(dj)\leq t\leq f(d(j+q))-1}}|e(f(dj)\kappa )-e(t\kappa )|\sum
_{t=f(ds)}^{n}\left |\frac{1}{n}-\frac{t}{n^{2}}\right |  \notag \\
&&+2\sum _{t=n+1}^{f(d(m_*+q))-1}\left |\frac{1}{n}-\frac{t}{n^{2}}\right
||e(f(dj)\kappa )|.  \notag
\end{eqnarray}
Using the previous two relations, (\ref{r:CD2}) and (\ref{r:CD4}), we
conclude that 
\begin{equation}
D-F_{n}^{\ast }(\kappa )=O_{f}(dqn^{-1/k})+O_{f}(dq|\kappa |n^{1-1/k}){%
\text {.}}  \label{r:CF}
\end{equation}

Now note that $S_{d}(af,q)=\sum _{s=0}^{q-1}e(f(ds)a/q)$ is real. The claim
now follows by combining (\ref{r:AB}), (\ref{r:BBCC}), (\ref{r:CCDD}) and (%
\ref{r:CF}).
\end{proof}

Let ${\mathfrak{M}}(Q,R)$ denote the major arcs, namely the set of all $x\in 
{\mathbb{R}}$ which can be approximated by a rational $a/q$, $(a,q)=1$,
where $q\leq Q$, so that $|x-a/q|\leq 1/qR$, and let ${\mathfrak{m}}(Q,R)={%
\mathbb{R}}\backslash {\mathfrak{M}}(Q,R)$ be the minor arcs. We also define
a function $\tau (d,q)$ which will describe the behavior of the principal
part of the major arcs estimate:

\begin{eqnarray*}
\tau ^{\ast }(d,q)&=&\left \{ 
\begin{array}{cc}
1, & q|d^{l} \\ 
-c_{0}|\alpha _l|r^{-1/k} & {\text {otherwise,}}%
\end{array}
\right . \\
\tau (d,q)&=&\max \{\tau ^{\ast }(d,q),-1\},
\end{eqnarray*}
where $r=q/(q,d^{l})$ and $c_{0}$ is the implicit constant in Lemma \ref%
{l:cn}. Lemma \ref{l:cn} and Lemma \ref{l:content} imply that $%
S_{d}(af,q)/q\geq \tau (d,q)$. As $F_{n}(\kappa )$ is non-negative, we can
now put all the results of this section together:

\begin{corollary}
\label{c:major}\textbf{The major arcs estimate.} Let $G_{n,d}(x)$ be a
trigonometric polynomial as in (\ref{r:defg}) for some integer polynomial $f$
of a degree $k\geq 3$ satisfying (\ref{a:polynomial}). Assume $1\leq Q<R$
are given. Let $x\in {\mathfrak{M}}(Q,R)$, $x=a/q+\kappa $, $(a,q)=1$, $%
q\leq Q$. Then 
\begin{equation*}
G_{n,d}(x)\geq \tau (d,q)F_{n}(\kappa )+O_{f}(dn^{-1/k}(Q+n/R)){\text{.}}
\end{equation*}
\end{corollary}

\section{The minor arcs}

We derive the following Lemma from the well-known estimates of Vinogradov.

\begin{lemma}
\label{l:vinogradov}Let $f(x)=\alpha _{k}x^{k}+...+\alpha _{1}x$ be an
integer polynomial. If $m,d,1\leq Q<R$ are constants so that $Q\geq \alpha
_{k}d^{k}m^{1/k}$ and $x\in {\mathfrak{m}}(Q,R)$, then 
\begin{equation*}
\sup _{1\leq m_{\ast }\leq m}\left \vert \sum _{j=1}^{m_{\ast
}}e(f(dj)x)\right \vert \ll _{f}\left (d^{-k}QR\right
)^{1/(k-1/k)}+m^{1-\rho },
\end{equation*}
where $\rho =1/(8k^{2}(\log k+1.5\log \log k+4.2))$.
\end{lemma}

\begin{proof}
We write $g(j)$ instead of $f(dj)x$, or more precisely: let $g(y)=\beta
_{k}y^{k}+...+\beta _{1}y$ where $\beta _{j}=\alpha _{j}d^{j}x$. Let $1\leq
m_{\ast }\leq m$. Applying Vinogradov exponential sum bounds (see \cite%
{Vinogradov:84}, Section 11), it is easy to see that if $Q\geq \alpha
_{k}d^{k}m_{\ast }^{1/k}$, $m_{\ast }\geq (\alpha
_{k}^{-1}d^{-k}QR)^{1/(k-1/k)}$ and $x\in {\mathfrak{m}}(Q,R)$, then ${%
\mathbf{\beta }}=(\beta _{k},...,\beta _{1})\,$\ is of the second class,
thus 
\begin{equation}
\left\vert \sum_{j=1}^{m_{\ast }}e(f(dj)x)\right\vert \ll _{f}m_{\ast
}^{1-\rho }\leq m^{1-\rho }.  \label{r:aa}
\end{equation}

Trivially for any $m_{\ast }$,%
\begin{equation}
\left\vert \sum_{j=1}^{m_{\ast }}e(f(dj)x)\right\vert \leq m_{\ast }.
\label{r:bb}
\end{equation}

The claim now follows by summing (\ref{r:bb}) for the cases $m_{\ast }\geq
\left( \alpha _{k}^{-1}d^{-k}QR\right) ^{1/(k-1/k)}$ and (\ref{r:bb})$\ $for
the cases $m_{\ast }\leq (\alpha _{k}^{-1}d^{-k}QR)^{1/(k-1/k)}$.
\end{proof}

\begin{proposition}
\label{p:minor}\textbf{The minor arcs estimate.} Let $G_{n,d}(x)$ be the
trigonometric polynomial as in (\ref{r:defg}) for some odd integer
polynomial $f$ and $1\leq Q<R$ constants. Let $x\in {\mathfrak{m}}(Q,R)$.
Also assume that $d\leq n^{1/k}$ and $\alpha _{k}d^{k}n^{1/k^{2}}\leq Q$.
Then 
\begin{equation*}
G_{n,d}(x)\ll _{f}n^{-1/k}\left (QR\right )^{1/(k-1/k)}+dn^{-\rho /k}{\text {%
.}}
\end{equation*}
\end{proposition}

\begin{proof}
Choose $m$ so that $\alpha _{k}d^{k}m^{k}\leq n<\alpha _{k}d^{k}(m+1)^{k}$.
Then 
\begin{equation}
\frac{d^{k}m^{k-1}}{n}+\frac{d^{2k}m^{2k-1}}{n^{2}}\ll _{f}dn^{-1/k}{\text {.%
}}  \label{r:minorhelp}
\end{equation}

We introduce the notation 
\begin{eqnarray*}
g(j)&=&\alpha _{k}kd^{k}j^{k-1}\left (\frac{1}{n}-\frac{\alpha _{k}(dj)^{k}}{%
n^{2}}\right ), \\
h(j)&=&\cos (2\pi f(dj)x)={\mathrm{Re}}\,e(f(dj)x){\text {.}}
\end{eqnarray*}
By partial integration, using the notation $\Delta g(j)=g(j+1)-g(j)$, $%
H(j)=\sum _{i=1}^{j}h(i)$, we get 
\begin{eqnarray*}
G_{n,d}(x)&=&\frac{2}{K}\sum _{j=1}^{m}g(j)h(j)=\frac{2}{K}\left
(g(m)H(m)-\sum _{j=1}^{m-1}\Delta g(j)H(j)\right ) \\
&\ll _{f}&\frac{1}{K}\left (\frac{d^{k}m^{k-1}}{n}+\frac{d^{2k}m^{2k-1}}{%
n^{2}}\right )\sup _{1\leq m_*\leq m}|H(m_*)|
\end{eqnarray*}

By (\ref{r:DD}) and as $d\leq n^{1/k}$, $1/K=O_{f}(1)$. As $Q\geq \alpha
_{k}d^{k}n^{1/k^{2}}$ and as by choice of $m$, $n^{1/k}\geq m$, we get $%
Q\geq \alpha _{k}d^{k}m^{1/k}$. We now combine Lemma \ref{l:vinogradov} and (%
\ref{r:minorhelp}): 
\begin{eqnarray*}
G_{n,d}(x) &\ll &_{f}dn^{-1/k}\left( \left( d^{-k}QR\right)
^{1/(k-1/k)}+n^{1/k-\rho /k}\right) \\
&\leq &n^{-1/k}\left( QR\right) ^{1/(k-1/k)}+dn^{-\rho /k}.
\end{eqnarray*}
\end{proof}

\section{Cancelling out the leading term}

Recall the definition of the functions $\tau ^{\ast }(d,q),$ $\tau (d,q)$ in
Section 2, estimating the principal part of the major arcs estimate. For
clarity of presentation, denote by $\alpha =c_{0}|\alpha _{l}|$, $\beta =1/k$%
, and then 
\begin{equation*}
\tau (d,q)=\left\{ 
\begin{array}{ll}
1 & \text{if }q|d^{l} \\ 
\max \{-\alpha r^{-\beta },-1\} & \text{otherwise,}%
\end{array}%
\right.
\end{equation*}%
where $r=q/(q,d^{l})$. We use in this section only the facts that $\alpha >0$%
, $0<\beta <1$.

\begin{theorem}
\label{t:averaging}\textbf{Averaging}. Assume $\delta >0$ is given. Then
there exist integer constants $s>0$, $1=d_{0}<d_{1}<...<d_{s}$, $d_{s}=O(%
\func{exp}(c_{1}\delta ^{-1/\beta })$, $c_{1}$ depending only on $\alpha
,\beta $, and a real constant $\lambda >0$ such that for any integer $q$, 
\begin{equation}
\frac{1}{\Lambda }\sum_{j=0}^{s}\lambda ^{j}\tau (d_{j},q)\geq -\delta {%
\text{,}}  \label{r:averaging}
\end{equation}%
where $\Lambda =1+\lambda +...+\lambda ^{s}$.
\end{theorem}

We first discuss the case when $q$ is a prime power $q=p^{k}$, which encodes
the key idea of this section. If $p$ is a prime, then 
\begin{equation*}
\tau ^*(p^{j},p^{k})=\left \{ 
\begin{array}{cc}
1, & k\leq jl \\ 
-\alpha p^{-\beta (k-jl)} & {\text {otherwise.}}%
\end{array}
\right .
\end{equation*}

\begin{lemma}
\label{l:av1}Say $p$ is a prime, and $\mu $ any real constant satisfying 
\begin{equation}
1>\mu \geq \frac{\alpha +1}{\alpha +p^{\beta }}{\text {.}}  \label{r:mu}
\end{equation}
Then for any positive integer constants $s,k$, 
\begin{equation}
\sum _{j=0}^{s}\mu ^{j}\tau ^*(p^{j},p^{k})\geq -\mu ^{s+1}/(1-\mu ){\text {.%
}}  \label{r:mu0}
\end{equation}
\end{lemma}

\begin{proof}
It can be easily deduced from (\ref{r:mu}) that 
\begin{eqnarray}
\mu p^{\beta l}&>&1,  \label{r:mu000} \\
-\frac{\alpha p^{\beta (l-1)}}{\mu p^{\beta l}-1}&\geq &-\frac{1}{1-\mu }.
\label{r:mu00}
\end{eqnarray}

Assume $m$ is the largest integer so that $m\leq s$, $ml<k$. Then for all $%
j\leq m$, $\tau ^{\ast }(p^{j},p^{k})=-\alpha p^{\beta (jl-k)}$. Denote the
left side of (\ref{r:mu0}) by $A_{s}$. We first apply $k\geq ml+1$, then (%
\ref{r:mu000}) and finally (\ref{r:mu00}): 
\begin{eqnarray*}
A_{m} &=&-\alpha \sum_{j=0}^{m}\mu ^{j}p^{\beta (jl-k)}\geq -\alpha
\sum_{j=0}^{m}\mu ^{j}p^{\beta l(j-m-1)}=-\alpha \mu ^{m}p^{-\beta
}\sum_{j=0}^{m}(\mu p^{\beta l})^{-j}\geq \\
&\geq &-\alpha \mu ^{m}p^{-\beta }\sum_{j=0}^{\infty }(\mu p^{\beta
l})^{-j}=-\frac{\alpha p^{\beta (l-1)}\mu ^{m+1}}{\mu p^{\beta l}-1}\geq -%
\frac{\mu ^{m+1}}{1-\mu }{\text{.}}
\end{eqnarray*}%
The case $m=s$ is now proved. If $m<s$, then for $m<j\leq s$, $\tau ^{\ast
}(p^{j},p^{k})=1$, thus 
\begin{equation*}
A_{s}=A_{m}+\sum_{j=m+1}^{s}\mu ^{j}\geq -\frac{\mu ^{m+1}}{1-\mu }%
+\sum_{j=m+1}^{s}\mu ^{j}=-\mu ^{s+1}/(1-\mu ){\text{.}}
\end{equation*}
\end{proof}

We now improve Lemma \ref{l:av1} and (\ref{r:mu}), so that also for small $p$%
, $\mu $ can be close to $1/2$.

\begin{lemma}
\label{l:av15}Say $p$ is a prime, $1>\mu >1/2$ and $a\geq 1$ an integer
satisfying 
\begin{equation}
p^{\beta al}\geq \frac{\alpha p^{-\beta }(1-\mu )+2\mu -1}{\mu (2\mu -1)}.
\label{r:mu25}
\end{equation}

Then for any positive integers $s,k$, 
\begin{equation}
\sum _{j=0}^{s}\mu ^{j}\tau (p^{aj},p^{k})\geq -\mu ^{s+1}/(1-\mu ){\text {.}%
}  \label{r:mu25b}
\end{equation}
\end{lemma}

\begin{proof}
We follow the steps of the proof of Lemma \ref{l:av1}, and first note that (%
\ref{r:mu25}) implies 
\begin{eqnarray}
\mu p^{\beta al}&>&1,  \label{r:mu270} \\
-\frac{\alpha p^{-\beta }+\mu p^{\beta al}-1}{\mu p^{\beta al}-1}&\geq &-%
\frac{\mu }{1-\mu }{\text {.}}  \label{r:mu27}
\end{eqnarray}

Denote the left side of (\ref{r:mu25b}) by $B_{s}$. Let $m$ be the largest
integer so that $m\leq s$, $aml<k$. Then for all $j\leq m$, $\tau
(p^{aj},p^{k})\geq -\alpha p^{\beta (ajl-k)}$. In the calculation below we
apply that and the following facts respectively: $\tau \geq -1$ for $j=m$; $%
k\geq aml+1$;\ then (\ref{r:mu270}) and finally (\ref{r:mu27}). We thus have%
\begin{eqnarray*}
B_{m} &\geq &-\alpha \sum_{j=0}^{m-1}\mu ^{j}p^{\beta (ajl-k)}-\mu ^{m}\geq
-\alpha p^{-\beta }\sum_{j=0}^{m-1}\mu ^{j}p^{\beta al(j-m)}-\mu ^{m} \\
&\geq &-\alpha \mu ^{m-1}p^{-\beta (al+1)}\sum_{j=0}^{m-1}(\mu p^{\beta
al})^{-j}-\mu ^{m} \\
&\geq &-\mu ^{m}\frac{\alpha p^{-\beta }+\mu p^{\beta al}-1}{\mu p^{\beta
al}-1}\geq -\frac{\mu ^{m+1}}{1-\mu }{\text{.}}
\end{eqnarray*}

The rest of the proof is analogous to the proof of Lemma \ref{l:av1}.
\end{proof}

We now set $\lambda =1/2^{\beta }$, and combine Lemmas \ref{l:av1} and \ref%
{l:av15} to find the prime power components of $d_{j}$ in\ Theorem \ref%
{r:averaging}.

\begin{lemma}
\label{l:av2}There exist a constant $c_{2}$ depending only on $\alpha ,\beta 
$, so that the following holds:\ for any positive integer $s$ and prime
number $p\leq 2^{s}$, there exist integers $0=a_{0}\leq a_{1}\leq ...\leq
a_{s}$ such that for any positive $k$, 
\begin{eqnarray}
\sum_{j=0}^{s}\lambda ^{j}\tau (p^{a_{j}},p^{k}) &\geq &-\frac{1+\alpha }{%
1-\lambda }\lambda ^{s+1}{\text{,}}  \label{r:hb} \\
p^{a_{s}} &<&2^{c_{2}s}.  \label{r:pmax}
\end{eqnarray}
\end{lemma}

\begin{proof}
We will distinguish small and large primes, and will apply below Lemma \ref%
{l:av1} for large, and Lemma \ref{l:av15} for small primes. Let $p_{\ast
}=p_{\ast }(\alpha ,\beta )$ be the smallest prime such that (\ref{r:mu})\
holds for $p=p_{\ast }$ and $\mu =\lambda =1/2^{\beta }$ (and then it holds
for all $p\geq p_{\ast }$). Let $a_{\ast }=a_{\ast }(\alpha ,\beta )$ be the
smallest integer so that (\ref{r:mu25}) holds for $p=2$, $a=a_{\ast }$ and $%
\mu =\lambda =1/2^{\beta }$ (and then it holds for all primes $p$ and the
same $a=a_{\ast }$). We distinguish two cases:

(i)\ Assume $p$ is small, i.e. $p<p_{\ast }$. Then we set $a_{j}=a_{\ast }j$%
. Because of definition of $a_{\ast }$, we can apply Lemma \ref{l:av15} and
get 
\begin{equation}
\sum _{j=0}^{s}\lambda ^{j}\tau (p^{a_{j}},p^{k})\geq -\frac{1}{1-\lambda }%
\lambda ^{s+1}{\text {.}}  \label{r:finalA}
\end{equation}

We also see that 
\begin{equation}
p^{a_{s}}<p_{\ast }^{a_{\ast }s}{\text {.}}  \label{r:finalB}
\end{equation}

(ii) Let $p$ be large, that means $p^{\ast }\leq p\leq 2^{s}$. We find an
integer $q$ so that $p_{\ast }^{q}\leq p<p_{\ast }^{q+1}$, and let $b,r\,\ $%
be the quotient and the remainder of dividing $s$ by $q$, thus $s=bq+r$. Let 
$a_{j}=\left \lfloor j/q\right \rfloor $, where $\left \lfloor
x\right
\rfloor $ is the largest integer not greater than $x$. First note
that the function $f(x)=(\alpha +x^{q})/(\alpha +x)^{q}$ is increasing for $%
x\geq 1$ (e.g. by differentiating). Now applying this, the definition of $%
p_{\ast }$ and $p\geq p_{\ast }^{q}$, we get 
\begin{equation}
\lambda ^{q}\geq \left (\frac{\alpha +1}{\alpha +p_{\ast }^{\beta }}\right
)^{q}\geq \frac{\alpha +1}{\alpha +p_{\ast }^{\beta q}}\geq \frac{\alpha +1}{%
\alpha +p^{\beta }}{\text {.}}  \label{r:mu3}
\end{equation}

Denote the right side of (\ref{r:hb}) with $C_{s}(k)$ and let $C_{s}^*(k)$
be the same sum with $\tau ^*$ instead of $\tau $. We can now apply Lemma %
\ref{l:av1} with $\mu =\lambda ^{q}=1/2^{\beta q}$, and get 
\begin{eqnarray}
C_{bq-1}^*(k)&=&\sum _{j=0}^{b-1}(1+\lambda +...+\lambda ^{q-1})\mu ^{j}\tau
^*(p^{j},p^{k})\geq  \notag \\
&\geq &-\frac{1+\lambda +...+\lambda ^{q-1}}{1-\mu }\mu ^{b}=-\sum
_{j=bq}^{\infty }\lambda ^{j}{\text {.}}  \label{r:b11}
\end{eqnarray}

We analyze two cases. Suppose $k\leq bl$. Then $\tau ^{\ast }(p^{f},p^{k})=1$%
. We use (\ref{r:b11}) and get 
\begin{equation*}
C_{s}^{\ast }(k)=C_{bq-1}^{\ast }(k)+\sum_{j=bq}^{s}\lambda ^{j}\geq
-\sum_{j=s+1}^{\infty }\lambda ^{j}=-\frac{1}{1-\lambda }\lambda ^{s+1}{%
\text{.}}
\end{equation*}%
Now assume $k>bl$. Then $\tau ^{\ast }(p^{b},p^{k})\leq 0$ and also for all $%
j\leq b-1$, $\tau ^{\ast }(p^{j},p^{k})=p^{-\beta }\tau ^{\ast
}(p^{j},p^{k-1})$. We now get from (\ref{r:b11}) that 
\begin{equation}
C_{bq-1}^{\ast }(k)=p^{-\beta }C_{bq-1}^{\ast }(k-1)\geq -p^{-\beta
}\sum_{j=bq}^{\infty }\lambda ^{j}{\text{.}}  \label{r:b12}
\end{equation}%
It is easy to deduce from (\ref{r:mu3}) that 
\begin{equation}
-p^{-\beta }\geq -\lambda ^{q}.  \label{r:b13}
\end{equation}%
As $\tau ^{\ast }(p^{b},p^{k})\geq -\alpha p^{-\beta }$, because of (\ref%
{r:b12}), (\ref{r:b13}) and finally $bq+q\geq s+1$, we get 
\begin{eqnarray*}
C_{s}^{\ast }(k) &=&C_{bq-1}^{\ast }(k)+\sum_{j=bq}^{s}\lambda ^{j}\tau
^{\ast }(p^{b},p^{k})\geq -p^{-\beta }\sum_{j=bq}^{\infty }\lambda
^{j}-\sum_{j=bq}^{s}\lambda ^{j}\alpha p^{-\beta }\geq \\
&\geq &-(1+\alpha )\sum_{j=bq+q}^{\infty }\lambda ^{j}\geq -\frac{1+\alpha }{%
1-\lambda }\lambda ^{s+1}{\text{.}}
\end{eqnarray*}%
As $C_{s}(k)\geq C_{s}^{\ast }(k)$, we see that (\ref{r:hb}) holds in both
cases. Finally, 
\begin{equation}
p^{a_{s}}=p^{b}<p_{\ast }^{b(q+1)}\leq p_{\ast }^{2s}{\text{.}}
\label{r:finalC}
\end{equation}

We get (\ref{r:pmax}) from (\ref{r:finalB}) and (\ref{r:finalC}), with $%
c_{2}=\max \{2,a_{\ast }\}\log _{2}p_{\ast }$.
\end{proof}

We now show why the left side of (\ref{r:averaging})\ can be reduced to
analysis of a prime factor.

\begin{lemma}
\label{l:av3}Say $d_{0},d_{1},...,d_{s}$ is a sequence of integers such that 
$d_{j}|d_{j+1}$. Then for each integer $q$, there exists a prime $p$ such
that for all $j$, 
\begin{equation}
\tau (d_{j},q)\geq \tau (p^{a_{j}},p^{k}){\text {,}}  \label{r:prim2}
\end{equation}
where $p^{a_{j}},p^{k}$ are the factors in the prime decomposition of $d_{j}$%
, $q$ respectively.
\end{lemma}

\begin{proof}
If $q=1$, then $k=0$, so both sides of (\ref{r:prim2})\ are equal to $1$.
Assume now that $q>1$. Let $m+1$ be the smallest index such that $%
q|d_{m+1}^{l}$ (if there is no such $m$, we set $m=s)$. If $m=0$, then $%
q|d_{j}^{l}$ for all $j$, so both sides of (\ref{r:prim2})\ are equal to $1$%
. In that case, we choose any prime $p$ in the decomposition of $q$.

Now say $1\leq m\leq s$, and let $r=q/(q,d_{m}^{l})$ and let $p$ be any
prime in the prime decomposition of $r$. For $j\geq m+1$, both sides of (\ref%
{r:prim2}) are equal to 1. For $j\leq m$, it is straightforward to check (%
\ref{r:prim2}).
\end{proof}

We now complete the proof of Theorem \ref{t:averaging}. Recall that $\lambda
=1/2^{\beta }$. Let $c_3$ be the largest of the constants $(1+\alpha
)/(1-\lambda )$ and $\alpha /\lambda $, and choose $s$ so that 
\begin{equation}
\lambda \delta /c_3\leq \lambda ^{s+1}\leq \delta /c_3.  \label{r:tm1}
\end{equation}
Let $\Lambda =1+\lambda +\cdots +\lambda ^s$, $m=2^s$ and $2=p_1<p_2<\ldots
<p_t$ be all the prime numbers between $1$ and $m$, and let $a_j^i$ be the
exponents constructed in Lemma \ref{l:av2}, associated to the prime $p_i$, $%
i=1,\ldots ,t$, $j=0,\ldots ,s$. We set 
\begin{equation*}
d_{j}=\prod _{i=1}^tp_i^{a_{j}^i}{\text {.}}
\end{equation*}
Let $p$ be the smallest prime number constructed in Lemma \ref{l:av3}. If $%
p\leq m$, then $p=p_{i}$, for some $i=1,\ldots ,t$. Now applying Lemma \ref%
{l:av3}, Lemma \ref{l:av2}, (\ref{r:tm1}) and $\Lambda \geq 1$, we deduce
that for any positive integer $q$, 
\begin{equation*}
\frac{1}{\Lambda }\sum _{j=0}^{s}\lambda ^{j}\tau (d_{j},q)\geq \frac{1}{%
\Lambda }\sum _{j=0}^{s}\lambda ^{j}\tau (p_i^{a_{j}^i},p_i^{k})\geq -\frac{%
1+\alpha }{(1-\lambda )\Lambda }\lambda ^{s+1}\geq -\delta {\text {.}}
\end{equation*}
Now assume that $p>m$. Then Lemma \ref{l:av3} and (\ref{r:tm1}) imply that 
\begin{equation*}
\frac{1}{\Lambda }\sum _{j=0}^{s}\lambda ^{j}\tau (d_{j},q)\geq -\alpha
p^{-\beta }\geq -\delta {\text {.}}
\end{equation*}
We deduce that (\ref{r:averaging}) holds. Now we estimate $d_s$. By (\ref%
{r:tm1}) and the definition of $m$, we get $m\leq (c_3/\delta )^{1/\beta }$
and thus 
\begin{equation}
s\leq \frac{1}{\log 2}\log (c_3/\delta )^{1/\beta }.  \label{r:tm2}
\end{equation}
The prime number theorem implies that $t\leq c_4(m/\log m)$, for some
constant $c_4$, so 
\begin{equation}
t\leq c_4\frac{(c_3/\delta )^{1/\beta }}{\log (c_3/\delta )^{1/\beta }}.
\label{r:tm3}
\end{equation}
Finally, by applying (\ref{r:pmax}), (\ref{r:tm2}) and (\ref{r:tm3}), we get
that $d_s\leq \func{exp}(c_2st)\leq \func{exp}(c_1\delta ^{-1/\beta })$,
where $c_1=c_2c_3^{1/\beta }c_4/\log 2$.

\section{Proof of Theorem $\protect\ref{t:main1}$}

We choose $\delta >0$, and find first, by applying the Theorem \ref%
{t:averaging}, the constants $d_{0},...,d_{s}$ and $\lambda >0$ such that
for all integers $q$, 
\begin{eqnarray}
\frac{1}{\Lambda }\sum_{j=0}^{s}\lambda ^{j}\tau (d_{j},q) &\geq &-\delta /2{%
\text{,}}  \label{r:FinalAve} \\
d_{s} &\leq &\func{exp}(c_{5}\delta ^{-k}){\text{,}}  \label{r:Finald}
\end{eqnarray}%
where $\Lambda =1+\lambda +\cdots \lambda ^{s}$ and $c_{5}$ depends on the
degree and the coefficients of the polynomial $f$. Let $c_{6}$ and $c_{7}$
be the implicit constants from Corollary \ref{c:major} and Proposition \ref%
{p:minor} respectively. To streamline the calculations below, we define $%
c_{8}=2(\alpha _{k}+1)\max \{c_{5},c_{6},c_{7}\}/c_{5}$ and $d_{\ast }=c_{8}%
\func{exp}(c_{5}\delta ^{-k})$. Then it is easy to check that 
\begin{equation}
\max \{c_{6},c_{7}\}(\alpha _{k}+1)d_{\ast }^{-1}\leq \frac{\delta }{2}
\label{r:nep1}
\end{equation}%
and that $d_{j}\leq d_{\ast }$ for all $j=1,...,s$. Compiling the
constraints and the error terms from Corollary \ref{c:major} and Proposition %
\ref{p:minor}, we see that it is now enough to choose the constants $n,Q,R$
so that: 
\begin{eqnarray*}
c_{6}d_{\ast }n^{-1/k}(Q+n/R) &\leq &\delta /2, \\
d_{\ast } &\leq &n^{1/k}, \\
\alpha _{k}d_{\ast }^{k}n^{1/k^{2}} &\leq &Q, \\
c_{7}\left( n^{-1/k}\left( QR\right) ^{1/(k-1/k)}+d_{\ast }n^{-\rho
/k}\right) &\leq &\delta /2{\text{,}}
\end{eqnarray*}%
where $\rho =1/(8k^{2}(\log k+1.5\log \log k+4.2)$. One can check\ using (%
\ref{r:nep1}) that the choice $n=d_{\ast }^{k^{8}}$, $Q=\alpha _{k}d_{\ast
}^{1.5k^{6}}$ and $R=d_{\ast }^{k^{8}-k^{7}+k^{5}-2.5k^{4}}$ satisfies all
these relations. We now define the cosine polynomial 
\begin{equation*}
T(x)=\delta +\frac{(1-\delta )}{\Lambda }\sum_{j=0}^{s}\lambda
^{j}G_{n,d_{j}}(x).
\end{equation*}

Clearly $T(0)=1$. Now for $x\in {\mathfrak{M}}(Q,R)$, Corollary \ref{c:major}%
, (\ref{r:FinalAve}) with the choice of constants above imply $T(x)\geq
\delta +(1-\delta )(-\delta /2-\delta /2)\geq 0$. Similarly for $x\in {%
\mathfrak{m}}(Q,R)$, $T(x)\geq 0$. Choose $m_{j}$ such that $\alpha
_{k}d_{j}^{k}m_{j}^{k}\leq n<\alpha _{k}d_{j}^{k}(m_{j}+1)^{k}$, for $%
j=0,\ldots ,s$, and let $m=\max \{m_{0},\ldots ,m_{s}\}$. For given $\delta
>0$, the largest non-zero coefficient of the polynomial $T$ is of the order
at most $N=P(d_{\ast }m)$. From (\ref{r:pol1}), we get $N=O_{f}(\func{exp}%
(c_{5}(k+k^{8})\delta ^{-k}))$, thus $\delta =O_{f}((\log N)^{-1/k})$.

\section{Proof of the lower bound}

The proof of the lower bound mimics the construction of I. Ruzsa in the case
of $f(x)=x^{2}$ (\cite{Ruzsa09}).

\begin{lemma}
\label{l:congruence}Let $k\geq 3$ be an odd integer, $\beta >0$ an integer,
and $p\equiv 1({\mathrm{mod}}\,k)$ a prime, $p>\beta $. Then there exists a
collection of integers $a_{1},a_{2},....,a_{s}$, $s=(p-1)/k$ such that for
any integer $j$, $(j,p)=1$, 
\begin{equation*}
\sum _{i=1}^{s}\cos (2\pi \beta j^{k}a_{i}/p)\leq -\sqrt{s/(k-2)}{\text {.}}
\end{equation*}
\end{lemma}

\begin{proof}
Assume without loss of generality that $\beta =1$ (we can do it as $(\beta
,p)=1$). As the congruence $x^{k}\equiv y^{k}({\mathrm{mod}}\,p)$ has $k$
solutions for any $y$ relatively prime with $p$, we can divide the set of $%
p-1$ reduced residue classes ${\mathrm{mod}}\,p$ into $k$ equivalence
classes $Q_{1},...,Q_{k}$ of size $s=(p-1)/k$, defined as: $a_{1}\sim a_{2}$
if for some $j$, $(j,p)=1$, 
\begin{equation*}
a_{1}a_{2}^{-1}\equiv j^{k}{\text{ }}({\mathrm{mod}}\,p)
\end{equation*}%
(the $a_{2}^{-1}$ is the multiplicative inverse of $a_{2}$ ${\mathrm{mod}}%
\,p $). As $k$ is odd, $a\sim -a$. We conclude that the sum $A_{m}$ defined
below is real, on the left-hand side independent of $j$, $(j,p)=1$ and on
the right-hand side independent of $a\in Q_{m}$: 
\begin{equation}
A_{m}=\sum_{a\in Q_{m}}e(j^{k}a/p)=\frac{1}{k}\sum_{j=1}^{p-1}e(j^{k}a/p){%
\text{.}}  \label{r:suma1}
\end{equation}

By definition, 
\begin{equation}
\sum _{m=1}^{k}A_{m}=\sum _{a,(a,p)=1}e(j^{k}a/p)=-1{\text {.}}
\label{r:one}
\end{equation}

We now evaluate $\sum_{m=1}^{k}A_{m}^{2}$ by using the right-hand side of (%
\ref{r:suma1}), and get 
\begin{eqnarray}
\sum_{m=1}^{k}sk^{2}A_{m}^{2} &=&\sum_{a=1}^{p-1}\left\vert
\sum_{x=1}^{p-1}e(x^{k}a/p)\right\vert
^{2}=\sum_{a=1}^{p-1}\sum_{x,y=1}^{p-1}e((x^{k}-y^{k})a/p)=  \notag \\
&=&\sum_{x,y=1}^{p-1}\sum_{a=1}^{p-1}e((x^{k}-y^{k})a/p)=(kp-p+1)(p-1), 
\notag \\
\sum_{m=1}^{k}A_{m}^{2} &=&p-s{\text{,}}  \label{r:two}
\end{eqnarray}%
where we used that $x^{k}\equiv y^{k}$ has $k$ solutions ${\mathrm{mod}}\,p$
for $x,y$ relatively prime with $p$. Now suppose all $A_{m}\geq -c$ for some 
$c\geq 0$. If there are $k^{-}$ numbers $A_{m}<0$, $1\leq k^{-}\leq k$, then 
$\sum_{m,A_{m}<0}A_{m}^{2}\leq k^{-}c^{2}$, and by using (\ref{r:one}), $%
\sum_{m,A_{m}\geq 0}A_{m}^{2}\leq (k^{-}c-1)^{2}$. Combining that with (\ref%
{r:two}), we easily get $c\geq \sqrt{s/(k-2)}$. Now we can find $A_{m}\leq -%
\sqrt{s/(k-2)}$, and choose $a_{1},...,a_{s}$ to be the elements of $Q_{m}$.
\end{proof}

We now complete the proof of Theorem \ref{t:lower}. Choose any cosine
polynomial (\ref{r:type}), $T(x)\geq 0$ and $T(0)=1$, with $f(x)=\beta x^{k}$%
. By calculating $\sum _{i=1}^{s}T(a_{i}/p)$, $p\equiv 1({\mathrm{mod}}\,k)$%
, $p>\beta $, applying Lemma \ref{l:congruence} and noting that $1+\sqrt{%
s(k-2)}\leq \sqrt{p}$, we get 
\begin{equation*}
\sum _{p|j}a_{j}\geq \frac{1}{\sqrt{p}}{\text {.}}
\end{equation*}

We multiply this by $\log p$ and sum for $p\leq m$. We get 
\begin{equation*}
\sum a_{j}\sum _{p|j,p\equiv 1({\mathrm{mod}}\,k),p\leq m}\log p\geq \sum
_{p|j,p\equiv 1({\mathrm{mod}}\,k),p\leq m}\frac{\log p}{\sqrt{p}}{\text {.}}
\end{equation*}

By the theorem on primes in arithmetic sequences, the sum on the right side
is $\sqrt{m}(2/\varphi (k)+o(1))$. On the left-hand side, the coefficient of 
$a_{j}$, $j>0$ is $\leq \log j\leq \log n$ (where $n$ is the largest
non-zero coefficient of $T(x)$), and the coefficient of $a_{0}$ is $%
m(1/\varphi (k)+o(1))$, where $\varphi (k)$ is the Euler's totient function.
By writing $\varepsilon =o(1)$, we get 
\begin{equation*}
ma_{0}\left( \frac{1}{\varphi (k)}+\varepsilon \right) +\log n\geq \left( 
\frac{2}{\varphi (k)}-\varepsilon \right) \sqrt{m}{\text{.}}
\end{equation*}

We express $a_{0}$, minimize over $\sqrt{m}$ and obtain $a_{0}\geq \left(
1/\varphi (k)-\varepsilon \right) /\log n$.

\begin{acknowledgement}
The authors thank I. Z. Ruzsa for suggesting the use of the approximate Fej%
\'{e}r kernel.
\end{acknowledgement}

\end{document}